\documentclass[12pt,reqno]{amsart}
\usepackage{amsmath,amssymb,amsthm, amsfonts}
\usepackage{mathpazo}          % Elegant Palatino font (optional)
\usepackage{hyperref}          % Clickable references
\usepackage{microtype}         % Improves text appearance
\usepackage{graphicx} % Required for inserting images
\usepackage{amsfonts}
\usepackage{amsmath}
\usepackage[utf8]{inputenc}
\usepackage[T1]{fontenc}
\usepackage{mathrsfs}
\usepackage{color}
\usepackage{hyperref}
\usepackage{comment}
\usepackage{enumerate}
\usepackage{enumitem}
\usepackage{pgfplots}
\usetikzlibrary{positioning}
\usetikzlibrary{calc}

\usepackage{mathrsfs}
\usepackage{latexsym}

\usepackage{subcaption}
\usepackage{tikz}
\usepackage[letterpaper,margin=1in]{geometry}
\newtheorem{defn}{Definition}[section]
\newtheorem{theorem}{Theorem}[section]

\newtheorem{lemma}{Lemma}[section]

\newcommand{\RR}{\mathbb{R}}
\newcommand{\NN}{\mathbb{N}}

\date{\today}
\title{On principal eigenpairs for the $(p,q)$-Laplacian in exterior domain}
\author[M.~Chhetri]{Maya Chhetri}
\address{M. Chhetri \newline
Department of Mathematics and Statistics, University of North Carolina Greensboro, Greensboro, NC, USA}
\email{m\_chhetr@uncg.edu}

\author[P.~Dr\'abek]{Pavel Dr\'abek}
\address{P.~Dr\'abek \newline
Departament of Mathematics and NTIS, Faculty of Applied Sciences, University of West Bohemia, Univerzitni 22 30614 Plzen, Czech Republic}
\email{pdrabek@kma.zcu.cz}

\author[R.~Shivaji]{Ratnasingham Shivaji}
\address{R.~Shivaji \newline
Department of Mathematics and Statistics, University of North Carolina Greensboro, Greensboro, NC, USA}
\email{shivaji@uncg.edu}

\dedicatory{Dedicated to the memory of Professor Stanislav I. Pohozaev}

\begin{document}

\begin{abstract}
\noindent
We consider an eigenvalue problem of the form
\begin{equation*}
    \left\{\begin{array}{rclll}
    -\Delta_{p} u  -\Delta_{q} u&=& \lambda K(x)|u|^{p-2}u &\mbox{ in } \Omega^e \\ 
     u&=&0\qquad \quad &\mbox{ on } \partial \Omega\\
     u(x) &\to& 0 &\mbox{ as }  |x| \to \infty\,,
     \end{array}\right.
\end{equation*}
where $\Omega^e$ is the exterior of an open connected and bounded set $\Omega$ in $\RR^N$  $(N \geq 2)$, $p, q \in (1, N)$ with $p \neq q$,  $0 < K \in L^{\infty}(\Omega^e) \cap L^{\frac{N}{p}}(\Omega^e)$, and $\lambda \in \RR$. We establish the existence of an unbounded set of the principal eigenvalues and corresponding eigenfunctions. Moreover, we establish the regularity, positivity  and the asymptotic profiles of these eigenfunctions with respect to the eigenvalue parameter $\lambda$. We use the {\em fibering method} of S.~I. Pohozaev  to prove our results.
\end{abstract}
\maketitle
\section{Introduction}
Let $\Omega \subset \RR^N\, (N \geq 2)$ be an open, connected, bounded set     
and $\Omega^e:=\RR^N \setminus \overline{\Omega}$. Let $p, q \in (1, N)$ with $p \neq q$ and $K: \Omega^e \to (0, \infty)$ be such that $K \in L^{\infty}(\Omega^e) \cap L^{\frac{N}{p}}(\Omega^e)$. Throughout this paper, we focus on the following eigenvalue problem:
\begin{equation}
\label{pde}
    \left\{\begin{array}{rclll}
    -\Delta_{p} u  -\Delta_{q} u&=& \lambda K(x)|u|^{p-2}u &\mbox{ in } \Omega^e \\ 
     u&=&0\qquad \quad &\mbox{ on } \partial \Omega\\
     u(x) &\to& 0 &\mbox{ as }  |x| \to \infty\,,
     \end{array}\right.
\end{equation}
where $\lambda \in \RR$ is the eigenvalue parameter.
\par For $r \in (1, N)$, let $L^r(\Omega^e)$  denote the Lebesgue space with respect to the norm $\|u\|_{L^r}:= \left(\int_{\Omega^e}|u|^r\right)^{\frac{1}{r}}$,  and $W_0^{1,r}(\Omega^e):=\overline{C_0^{\infty}(\Omega^e)}^{\|u\|_r}$ denotes the Sobolev space, where the closure is taken with respect to $\|u\|_r:=\||\nabla u|\|_{L^r}$.

\par We will be utilizing an Orlicz space and the corresponding  Orlicz-Sobolev space to study the problem \eqref{pde}. To discuss these, for
$$
\psi_{p, q}(t):=|t|^{p-2}t + |t|^{q-2}t \mbox { for } p,q > 1 \text{ and } t \in \RR\,,
$$
we introduce the corresponding $N$-function
\begin{equation}
\label{def:N:fn}
\Psi_{p,q}(t):=\int_0^t\psi_{p,q}(s)\mathrm{d}s = \frac{t^p}{p} + \frac{t^q}{q} \mbox{ for } t \geq 0\,.
\end{equation}
We define the Orlicz space 
$$
Y:=L^{\Psi_{p,q}}(\Omega^e)=\left\{u: \Omega^e \to \RR \mbox{ measurable}:   \int_{\Omega^e} \Psi_{p,q}(|u|) < \infty\right\}\,.
$$
The space $Y$ is a Banach space with respect to the {\em Luxemburg norm} 
$$
\|u\|_Y:= 
\inf\left\{
\alpha >0: \int_{\Omega^e}\Psi_{p,q}\left(\frac{|u(x)|}{\alpha}\right) \leq 1
\right\}
$$
and the corresponding Orlicz-Sobolev space is given by
$X:=W_0^{1,\Psi_{p,q}}(\Omega^e):=\overline{C_0^{\infty}(\Omega^e)}^{\|u\|_X}
$, where $\|u\|_X:=\||\nabla u|\|_Y$. We remark that $Y$ is a uniformly convex Banach space, and hence so is $X$. See e.g.,  \cite[Cor.~6.11]{KZ2022} and \cite[Thm.~3]{Chen-Hu-Zhao-2003} for details. Further, 
 $$
X \hookrightarrow W_0^{1,p}(\Omega^e) \mbox{ and } 
X \hookrightarrow W_0^{1,q}(\Omega^e)\,. 
$$
In particular, 
\begin{equation}
\label{est:Sobolev:norm:bound}
\|u\|_p \leq p^{\frac{1}{p}}\|u\|_X \mbox{ and } \|u\|_q \leq q^{\frac{1}{q}}\|u\|_X\,,
\end{equation}
see \cite{Grecu2022} and \cite{Mihai-Stan-2016} for details. On the other hand, it follows from \eqref{def:N:fn} that 
\begin{equation}
    \label{est:X:norm:bound}
    \int_{\Omega^e}\Psi_{p,q}(|\nabla u|) = \frac{\|u\|_p^p}{p} + \frac{\|u\|_q^q}{q}\,.
\end{equation}
That is, if a function is bounded in $W^{1,p}_0(\Omega^e)$ and in $W^{1,q}_0(\Omega^e)$, then the function is bounded in $X$ as well since the Luxemburg norm is defined using the modular given by \eqref{def:N:fn}. As in \cite{Grecu2022, Mihai-Stan-2016}, the unbounded domain considered in \eqref{pde} gives rise to additional difficulties in the choice of appropriate function spaces for solutions, difficulties that do not arise in the bounded domain setting, where it is possible to work in the framework of Sobolev spaces. 
\begin{defn}
\label{def:soln:weak}
   A function $u \in X$ is a weak solution of \eqref{pde} if 
$$
\int_{\Omega^e}\left(|\nabla u|^{p-2} + |\nabla u|^{q-2}\right)\nabla u \nabla \varphi 
=\lambda \int_{\Omega^e} K(x) |u|^{p-2}u \varphi \quad \forall \varphi \in X\,.
$$
\end{defn}
\medskip
 Next, for $\lambda \in \RR$,   let  $J_{\lambda}: X \to \mathbb{R}$ be the energy functional defined by 
\begin{equation}
    \label{energy}
    J_{\lambda}(u):= \frac{1}{p}\int_{\Omega^e}|\nabla u|^p + 
     \frac{1}{q}\int_{\Omega^e}|\nabla u|^q -  \frac{\lambda}{p}\int_{\Omega^e}K(x)|u|^p\,.
\end{equation}
   Then $J_{\lambda}$ is of class $C^1$ with the derivative given by 
\begin{equation}
\label{deri:functional}
 \left<J_{\lambda}'(u), \phi\right> =\int_{\Omega^e}\left(|\nabla u|^{p-2} + |\nabla u|^{q-2}\right)\nabla u \nabla \phi 
-\lambda \int_{\Omega^e} K(x) |u|^{p-2}u \phi\,,  
\end{equation} 
for all $u,\, \phi \in X$, 
and critical points of $J_{\lambda}$ correspond to the weak solutions of \eqref{pde}. 

\medskip

Now, let 
\begin{equation}
\label{def:eigen}
    \lambda_1(p):=\inf\limits_{\substack{ u \in W_0^{1,p}(\Omega^e)\\
    u \neq 0}}\frac{\int_{\Omega^e}|\nabla u|^p}{\int_{\Omega^e}K(x)|u|^p}
\end{equation}
denote the principal eigenvalue and $\varphi_1$ the corresponding principal eigenfunction of $-\Delta_p$ in $\Omega^e$, see \cite{CD2014}. 
Then our main result reads as follows:
\begin{theorem}
\label{theo:main}
Problem \eqref{pde} admits a nontrivial weak solution 
 $u_{\lambda} \in X \cap L^{\infty}(\Omega^e)$ such that $u_{\lambda} >0$ in $\Omega^e$ and $\displaystyle \lim\limits_{|x| \to \infty}u_{\lambda}(x)=0$ uniformly if and only if 
 $\lambda \in (\lambda_1(p), \, \infty)$.  Further,
\begin{enumerate}[label=(\alph*)]
    \item if $p<q$ then 
    \begin{equation*}
        \lim\limits_{\lambda \to \lambda_1(p)^+}J_{\lambda}(u_{\lambda})= 0 \quad \text{ and }\quad \lim\limits_{\lambda \to \lambda_1(p)^+}\|u_{\lambda}\|_q = 0; \text{ and } 
        \end{equation*}        
        \begin{equation*}
        \lim\limits_{\lambda \to \infty}J_{\lambda}(u_{\lambda}) = -\infty \quad \text{ and }\quad \lim\limits_{\lambda \to \infty}\|u_{\lambda}\|_q = \infty\,,
        \end{equation*}       
     \item if $p>q$ then  \begin{equation*}
        \lim\limits_{\lambda \to \lambda_1(p)^+}J_{\lambda}(u_{\lambda})= \infty \quad \text{ and }\quad \lim\limits_{\lambda \to \lambda_1(p)^+}\|u_{\lambda}\|_q = \infty; \text{ and } 
        \end{equation*}        
        \begin{equation*}
        \lim\limits_{\lambda \to \infty}J_{\lambda}(u_{\lambda}) =0 \quad \text{ and }\quad  \lim\limits_{\lambda \to \infty}\|u_{\lambda}\|_q = 0\,.
        \end{equation*}       
\end{enumerate} 
Moreover, if $\partial \Omega$ is of class $C^2$, then $u_{\lambda} \in C_0^1(\overline{\Omega^e})$ and  $\frac{\partial u_{\lambda}}{\partial \eta} < 0$ on $\partial \Omega$, where $\eta$ is outer normal to $\partial \Omega$.
\end{theorem}
%%%%%%%%%%%%%%%%%%%%%%%%%%%%%%%%%%%%%%%%%%%%%%%
\bigskip
\begin{figure}[ht]
\centering

% ---- J_lambda FIRST ----
\begin{minipage}[t]{0.48\textwidth}
\centering
\begin{tikzpicture}[scale=1.1]
\useasboundingbox (0,0) rectangle (5,4);

\draw[->] (0,3.5) -- (5,3.5) node[right] {$\lambda$};
\draw[->] (0,0) -- (0,4) node[above] {$J_{\lambda}(u_{\lambda})$};

\draw[dashed] (1,3.5) -- (1,0);
\draw[thick, domain=1:5, samples=100]
    plot (\x,{ -sqrt(\x-1) +3.5});

\fill (1,3.5) circle (2pt);
\node[above right] at (1,3.5) {$\lambda_1(p)$};
\end{tikzpicture}
\end{minipage}
\hfill
% ---- ||u||_q SECOND ----
\begin{minipage}[t]{0.48\textwidth}
\centering
\begin{tikzpicture}[scale=1.1]
\useasboundingbox (0,0) rectangle (5,4);

\draw[->] (0,0) -- (5,0) node[right] {$\lambda$};
\draw[->] (0,0) -- (0,4) node[above] {$\|u_{\lambda}\|_q$};

\draw[dashed] (1,0) -- (1,4);
\draw[thick, domain=1:5, samples=100]
    plot (\x,{sqrt(\x-1)});

\fill (1,0) circle (2pt);
\node[above right] at (1.1,0) {$\lambda_1(p)$};
\end{tikzpicture}
\end{minipage}

\caption{Qualitative behaviors of $J_{\lambda}(u_{\lambda})$ and $\|u_{\lambda}\|_q$
as functions of $\lambda$ when $p<q$.}
\label{fig:p<q}
\end{figure}
%%%%%%%%%%%%%%%%%%%%%%%%
\bigskip
\begin{figure}[ht]
\centering

% ---- J_lambda FIRST ----
\begin{minipage}[t]{0.48\textwidth}
\centering
\begin{tikzpicture}[scale=1.1]
\useasboundingbox (0,0) rectangle (5,4);

\draw[->] (0,0) -- (5,0) node[right] {$\lambda$};
\draw[->] (0,0) -- (0,4) node[above] {$J_{\lambda}(u_{\lambda})$};

\draw[dashed] (1,0) -- (1,4);
\begin{scope}
  \clip (0,0) rectangle (5,4);
  \draw[thick, domain=1.05:5, samples=200]
      plot (\x,{1/(\x-1)});
\end{scope}

\fill (1,0) circle (2pt);
\node[above] at (1.6,0) {$\lambda_1(p)$};
\end{tikzpicture}
\end{minipage}
\hfill
% ---- ||u||_q SECOND ----
\begin{minipage}[t]{0.48\textwidth}
\centering
\begin{tikzpicture}[scale=1.1]
\useasboundingbox (0,0) rectangle (5,4);

\draw[->] (0,0) -- (5,0) node[right] {$\lambda$};
\draw[->] (0,0) -- (0,4) node[above] {$\|u_{\lambda}\|_q$};

\draw[dashed] (1,0) -- (1,4);
\begin{scope}
  \clip (0,0) rectangle (5,4);
  \draw[thick, domain=1.05:5, samples=200]
      plot (\x,{1/(\x-1)});
\end{scope}

\fill (1,0) circle (2pt);
\node[above] at (1.6,0) {$\lambda_1(p)$};
\end{tikzpicture}
\end{minipage}

\caption{Qualitative behaviors of $J_{\lambda}(u_{\lambda})$ and $\|u_{\lambda}\|_q$
as functions of $\lambda$ when $p>q$.}
\label{fig:p>q}
\end{figure}

%%%%%%%%%%%%%%%%%%%%%%%%%%%%%%%%%%%%%

We note that $(\lambda, u_{\lambda})$ can be regarded as the principal eigenpair (i.e., the principal eigenvalue and the corresponding principal eigenfunction) of the nonhomogeneous eigenvalue problem \eqref{pde}. Figure~\ref{fig:p<q} and Figure~\ref{fig:p>q} summarize the nonexistence, existence results, and the asymptotic properties described in Theorem~\ref{theo:main}.

\par For results on the existence, nonexistence, and regularity of solutions in bounded domains, we refer the reader to \cite{Bob-Tana_Calc-var, Bob-Tan-2018, Bob-Tan-2022}. See \cite{Mihai-Stan-2016} for existence and nonexistence of weak solutions for the $(2, q)$ - Laplacian in a general open set in $\RR^N$, which was later extended in \cite{Grecu2022} for the $(p, q)$ - Laplacian case in exterior domains. However, in both \cite{ Mihai-Stan-2016} and \cite{Grecu2022}, the positivity of the solution, the regularity, or the asymptotic behavior of the solution when $\lambda \to \lambda_1(p)^+$ and $\lambda \to \infty$ were not investigated. These are among the primary focuses addressed in this paper. In fact, our proofs of the asymptotic behavior of the solution as  $\lambda \to \lambda_1(p)^+$ and  $\lambda \to \infty$, can be carried over to  bounded domains, thus further contributing to the literature in this case as well.

\par We may also view problem \eqref{pde} from the point of bifurcation with the term $-\Delta_qu$ as the perturbation of the eigenvalue problem $-\Delta_p u = \lambda K(x)|u|^{p-2}u$. In particular, for $p<q$,  the perturbation represents the "higher order term" and the $(\lambda, \|u\|_q)$ branch
in Fig.~\ref{fig:p<q} corresponds to the bifurcation from zero. On the other hand, 
for $p>q$, the perturbation represents  the "lower order term" and the $(\lambda, \|u\|_q)$ branch
in Fig.~\ref{fig:p>q} corresponds to the bifurcation from infinity. This interpretation is in alignment with the classical bifurcation theory.

\par To establish our results, we use the \emph{fibering method} of \cite{Drabek-Pohozaev}, which goes back to the seminal works of Pohozaev; see, for instance, \cite{Pohojaev-1979,  pohozaev-99,   pohozaev-97, Pohojaev-1988, Pohozaev-1990}. The fibering functional and the fibering method are also mentioned in \cite{Bob-Tana_Calc-var, Bob-Tan-2018, Bob-Tan-2022}, in connection with the study of $(p, q)$- Laplacian problems in bounded domain case. This method provides a unified approach to deal with our variational problem, which involves a functional that may or may not be coercive. Further, compared to those  available in the current literature for $(p,q)$ - Laplacian, which are based on the notion of the \emph{Nehari} manifold, the fibering method yields explicit information on the dependence of both the energy and the norm of the solutions in $W^{1,q}_0(\Omega^e)$ as $\lambda \to \lambda_1(p)^+$ and  $\lambda \to \infty$, a main focus of this paper.

%%%%%%%%%%%%%%%%%%%%%%%%
\smallskip
The paper is organized as follows. 
In Section~\ref{sec:fiber}, we recall the fibering method and introduce a fibered functional whose conditional critical points on a suitable manifold correspond to critical points of the functional associated with~\eqref{pde}.
The proof of Theorem~\ref{theo:main} is then developed over Sections~\ref{sec:fiber}--\ref{sec:qual:prop}.  Sections~\ref{sec:p<q} and~\ref{sec:p>q} are devoted to establishing the existence of minimizers of this fibered functional in the cases $p<q$ and $p>q$, respectively. In Sections~\ref{sec:asympt} and~\ref{sec:qual:prop}, we derive the asymptotic behavior and qualitative (regularity) properties of the corresponding solutions, completing the proof of Theorem~\ref{theo:main}.

\par Finally, in Section~\ref{sec:final}, we discuss the advantages of using the fibering method approach here compared to the method based on the notion of the {\em Nehari} manifold used in the recent literature.

\section{Fibering method}
\label{sec:fiber}
We begin by discussing facts that will be essential for the subsequent analysis and estimates. The infimum in \eqref{def:eigen} is achieved at $\varphi_1 \in W^{1,p}_0(\Omega^e)$. Since $X \subsetneq W^{1,p}_0(\Omega^e)$ the principal eigenfunction $\varphi_1$ need not belong to $X$. However, \eqref{def:eigen} implies that 
\begin{equation}
    \label{char:lambda:1:p}
    \int_{\Omega^e}|\nabla u|^p \geq \lambda_1(p) \int_{\Omega^e}K(x)|u|^p \quad \mbox{for all}\quad u \in X\,,
\end{equation}
since it holds for all $u \in W^{1,p}_0(\Omega^e)$. 
Further, due to \cite[Lem.~2]{Grecu2022}, $\lambda_1(p)$ also has the following useful variational characterization.
\begin{lemma}
    \label{lem:Grecu:lem2}
Define
\begin{equation}
\label{eval:1:grecu}
\mu_1(p,q):=\inf\limits_{\substack{ u \in X \\
    u \neq 0}}\frac{\frac{1}{p}\int_{\Omega^e}|\nabla u|^p + \frac{1}{q}\int_{\Omega^e}|\nabla u|^q}{\frac{1}{p}\int_{\Omega^e}K(x)|u|^p}\,.
\end{equation}
Then $\mu_1(p,q)=\lambda_1(p)$.
\end{lemma}
As a consequence of Lemma~\ref{lem:Grecu:lem2},   given $\lambda > \lambda_1(p)$, there exists $\varphi_{\lambda} \in X$ such that 
\begin{equation}
    \label{char:lambda:1:p:reverse}
    \int_{\Omega^e}|\nabla \varphi_{\lambda}|^p < \lambda\int_{\Omega^e}K(x)|\varphi_{\lambda}|^p\,.
\end{equation}

\medskip

\par We employ the unified framework based on the \emph{fibering method}, as developed in \cite{Drabek-Pohozaev}, to analyze the functional $J_{\lambda}$. For this, defining $H_{\lambda}, G : X \to \RR$ by 
\begin{equation}
    \label{H}
    H_{\lambda}(u):= \int_{\Omega^e}|\nabla u|^p - \lambda\int_{\Omega^e}K(x)|u|^p
\end{equation}
and 
\begin{equation}
    \label{G}
   G(u):= \int_{\Omega^e}|\nabla u|^q
\end{equation}
respectively, we have 
\begin{equation}
\label{J:H:G}
    J_{\lambda}(u) = \frac{1}{p}H_{\lambda}(u) + \frac{1}{q}G(u) \quad \mbox{ for all}\quad u \in X\,.
\end{equation}
%%%%%%%%%%%%%%%%
\medskip

It follows from \eqref{G} that $G(u)>0$, for $u \neq 0$, and hence the sign  of $H_{\lambda}$ given by \eqref{H} plays a crucial role to study $J_{\lambda}$ defined in \eqref{J:H:G}. For this purpose, we define the sets
\begin{equation}
    \label{set:H:minus}
\mathcal{H}_{\lambda}^-:= \left\{u \in X \setminus \{0\}: H_{\lambda}(u) <0 \right\}
\end{equation}
and 
\begin{equation}
    \label{set:H:plus}
\mathcal{H}_{\lambda}^+:= X \setminus \mathcal{H}_{\lambda}^- =  \left\{u \in X: H_{\lambda}(u) \geq 0 \right\}\,.
\end{equation}
To discuss the {\em fibering method}, we represent $ u \in X \setminus\{0\}$ in the form:
\begin{equation}
\label{split}
u(x)=tv(x)    \quad \mbox{for}\quad t \in \RR\,,v \in X\, \mbox{ with } t \neq 0, \, v \neq 0\,.
\end{equation}
Substituting \eqref{split} into \eqref{energy}, we get the \emph{fibered functional}
\begin{equation}
\label{fibrer}
    J_{\lambda}(tv) = 
    \frac{|t|^p}{p}\int_{\Omega^e}|\nabla v|^p + 
     \frac{|t|^q}{q}\int_{\Omega^e}|\nabla v|^q -  \lambda\frac{|t|^p}{p}\int_{\Omega^e}K(x)|v|^p\,.
\end{equation}
Let $u \in X\setminus\{0\},\ u=tv$, be a nonzero critical point of $J_{\lambda}$. Then, necessarily 
\begin{align}
0=& \frac{\partial}{\partial t}J_{\lambda}(tv) \nonumber\\
= &  |t|^{p-2}t\int_{\Omega^e}|\nabla v|^p + 
      |t|^{q-2}t\int_{\Omega^e}|\nabla v|^q -  \lambda|t|^{p-2}t\int_{\Omega^e}K(x)|v|^p \nonumber\\
      = & |t|^{p-2}tH_{\lambda}(v) +  |t|^{q-2}t G(v)\,.
      \label{derivative:fiber:1}
\end{align}

\smallskip

Let $\lambda \leq \lambda_1(p)$. Then by \eqref{char:lambda:1:p},
we have $\mathcal{H}_{\lambda}^- = \emptyset$ and due to the fact that $H_{\lambda}(v) \geq 0$, $G(v) >0$, \eqref{derivative:fiber:1} cannot hold for any $t \neq 0$. Therefore, $J_{\lambda}$ does not have a nonzero critical point.

\smallskip

\par On the other hand, if $\lambda > \lambda_1(p)$ then it follows from \eqref{char:lambda:1:p:reverse} that $\varphi_{\lambda} \in \mathcal{H}_{\lambda}^-$ and thus the set $\mathcal{H}_{\lambda}^-$ is nonempty. Then for any $v \in \mathcal{H}_{\lambda}^-$, 
\eqref{derivative:fiber:1} has a solution $t_v$ satisfying 
\begin{equation}
\label{def:t(v)}
 |t_v|:=\left(\frac{|H_{\lambda}(v)|}{G(v)}\right)^{\frac{1}{q-p}} >0\,.
\end{equation}
Substituting \eqref{def:t(v)} into \eqref{fibrer}, and using the fact that $v \in \mathcal{H}_{\lambda}^-$, the fibered functional takes the form:
\begin{align}
        \widetilde{J}_{\lambda}(v):=J_{\lambda}(t_vv)&= \frac{|t_v|^p}{p}H_{\lambda}(v) + \frac{|t_v|^q}{q}G(v) \nonumber\\
    &= -\frac{1}{p}\left(\frac{|H_{\lambda}(v)|}{G(v)}\right)^{\frac{p}{q-p}}|H_{\lambda}(v)| +  \frac{1}{q}\left(\frac{|H_{\lambda}(v)|}{G(v)}\right)^{\frac{q}{q-p}}G(v) \nonumber\\
    &= \left(\frac{1}{q}-\frac{1}{p}\right)\frac{|H_{\lambda}(v)|^{\frac{q}{q-p}}}{ [G(v)]^{\frac{p}{q-p}}}\,.
    \label{fiber:crit:val}
\end{align}
We observe that the fibered functional  $\widetilde{J}_{\lambda}$ has the following properties:
\begin{enumerate}[label=(\roman*)]
    \item it is $0$-homogeneous, i.e., for every $v \in \mathcal{H}_{\lambda}^-$ and every $t \in \RR \setminus\{0\}$, we have
    \begin{equation*}
        \widetilde{J}_{\lambda}(tv)=\widetilde{J}_{\lambda}(v)\,;
    \end{equation*}
    \item it is even and the Gateaux derivative of $\widetilde{J}_{\lambda}$ at the point $v \in \mathcal{H}_{\lambda}^-$ in the direction of $v$ is zero, i.e.,
    \begin{equation*}
        \left< \widetilde{J}_{\lambda}'(v), v\right> =0\,;
    \end{equation*}
    \item it follows from \eqref{fiber:crit:val} that if $v_{\lambda} \in \mathcal{H}_{\lambda}^-$ is a minimizer of $\widetilde{J}_{\lambda}$ then also $|v_{\lambda}| \in \mathcal{H}_{\lambda}^-$ is a minimizer of $\widetilde{J}_{\lambda}$ (cf. Section~\ref{sec:qual:prop}). In particular, we can automatically assume that $v_{\lambda}$ is nonnegative in $\Omega^e$.
\end{enumerate}

\medskip
\par The following lemma clarifies the relationship between the critical points of the functional $J_{\lambda}$ and the  fibered functional $\widetilde{J}_{\lambda}$.
\begin{lemma}
\label{lem:c.p:fiber}
    Every critical point $v_{\lambda} \in \mathcal{H}_{\lambda}^-$  of $\widetilde{J}_{\lambda}$ generates the critical point $u_{\lambda} \in X$ of $J_{\lambda}$ given by the formula $u_{\lambda}(x)=t_{v_{\lambda}}v_{\lambda}(x)$, where $t_{v_{\lambda}} \neq 0$ is given by \eqref{def:t(v)}.
\end{lemma}
\begin{proof}
    Let $v_{\lambda} \in \mathcal{H}_{\lambda}^-$ be a critical point of   $\widetilde{J}_{\lambda}$. Then 
    \begin{equation*}
        \widetilde{J}_{\lambda}'(v_{\lambda})=0\,.
    \end{equation*}
Then, for any $\phi \in X$, it follows from \eqref{def:t(v)} and \eqref{fibrer} that  
\begin{align*}
0= \left<\widetilde{J}_{\lambda}'(v_{\lambda}), \phi \right>
= &\left(\frac{1}{q} - \frac{1}{p}\right) 
\left[
\frac{-q}{q-p}|H_{\lambda}(v_{\lambda})|^{\frac{p}{q-p}}G(v_{\lambda})^{\frac{-p}{q-p}}
\left<H_{\lambda}'(v_{\lambda}), \phi\right>  \right. \\
&\qquad - \left.\frac{p}{q-p}|H_{\lambda}(v_{\lambda})|^{\frac{q}{q-p}}G(v_{\lambda})^{\frac{-q}{q-p}}
\left<G'(v_{\lambda}), \phi\right>
\right]  \\
= & \frac{1}{p}\left(\frac{|H_{\lambda}(v_{\lambda})|}{G(v_{\lambda})}\right)^{\frac{p}{q-p}}\left<H_{\lambda}'(v_{\lambda}), \phi\right> + \frac{1}{q}\left(\frac{|H_{\lambda}(v_{\lambda})|}{G(v_{\lambda})}\right)^{\frac{q}{q-p}}\left<G'(v_{\lambda}), \phi\right>  \\
= & |t_{v_{\lambda}}|
\left[
\frac{1}{p}|t_{v_{\lambda}}|^{p-1} \left<H_{\lambda}'(v_{\lambda}), \phi\right> +
\frac{1}{q}|t_{v_{\lambda}}|^{q-1}\left<G'(v_{\lambda}), \phi\right> 
\right] \nonumber \\
= &  |t_{v_{\lambda}}| \left<{J}_{\lambda}'(t_{v_{\lambda}}v_{\lambda}), \phi\right>\,, 
\end{align*}
that is, ${J}_{\lambda}'(u_{\lambda})=0$ since $t_{v_{\lambda}} \neq 0$. This completes the proof.
\end{proof}
Next we show that the critical point $v_{\lambda}$ of $\widetilde{J}_{\lambda}$ can be found using the conditional variational problem associated with $\widetilde{J}_{\lambda}$.
\begin{lemma}
    \label{lem:cond:c.p.} 
Let us consider the constraint set 
\begin{equation}
\label{set:W:lambda}    
\mathcal{W}_{\lambda}:= \left\{ v \in \mathcal{H}_{\lambda}^-: \, G(v)=1\right\}\,.
\end{equation} 
Then every conditional critical point of the functional $\widetilde{J}_{\lambda}$ restricted to $\mathcal{W}_{\lambda}$ is a usual critical point of $\widetilde{J}_{\lambda}$.
\end{lemma}
\begin{proof}
    Let $v_{\lambda}$ be the conditional critical point of  $\widetilde{J}_{\lambda}$ restricted to $\mathcal{W}_{\lambda}$. It follows from the Lagrange multiplier method that there exist real numbers $\mu_i$, $i=1,2$ such that 
    \begin{equation}
        \label{lagrange:1}
        \mu_1 \widetilde{J}_{\lambda}'(v_{\lambda}) = \mu_2 G'(v_{\lambda})\,,
    \end{equation}
    where
    \begin{equation}
    \label{lagrange:2}
        \mu_1^2 + \mu_2^2 >0\,.
    \end{equation}
    It follows from \eqref{lagrange:1} that 
    \begin{equation}
        \label{lagrange:3}
        \mu_1 \left<\widetilde{J}_{\lambda}'(v_{\lambda}), v_{\lambda} \right> = \mu_2 \left<G'(v_{\lambda}), v_{\lambda} \right>\,.
    \end{equation}
But $\left<\widetilde{J}_{\lambda}'(v_{\lambda}), v_{\lambda} \right>=0$ by property (ii) of $\widetilde{J}_{\lambda}$ while  $\displaystyle\left<G'(v_{\lambda}), v_{\lambda} \right> = q \int_{\Omega^e}|\nabla v_{\lambda}|^q = q\neq 0$. Then it follows from \eqref{lagrange:2} and \eqref{lagrange:3} that $\mu_1 \neq 0$ and $\mu_2=0$.  Hence \eqref{lagrange:3} implies that $\widetilde{J}_{\lambda}'(v_{\lambda}) =0$, completing the proof.
\end{proof}
\par Combining Lemma~\ref{lem:c.p:fiber} and Lemma~\ref{lem:cond:c.p.}, we will look for the existence of minimizer of $\widetilde{J}_{\lambda}$
subject to the constraint $\mathcal{W}_{\lambda}$. For this purpose, we have to distinguish between the following two qualitatively different cases: $p < q$ and $p > q$.
\par We note that the fibered functional $\widetilde{J}_{\lambda}$ in \eqref{fiber:crit:val}, restricted to the constraint $\mathcal{W}_{\lambda}$, given by \eqref{set:W:lambda}, takes the form 
\begin{equation}
    \label{fiber:on:W:lambda}
    \widetilde{J}_{\lambda}(v)=\left(\frac{1}{q} - \frac{1}{p}\right)|H_{\lambda}(v)|^{\frac{q}{q-p}} \qquad \mbox{for}\quad v \in \mathcal{W}_{\lambda}\,.
\end{equation}
\section{The case: $p < q$}
\label{sec:p<q}
\begin{lemma}
    \label{lem:p<q}
    Let $p< q$ and $\lambda > \lambda_1(p)$. 
Then there exists $v_{\lambda} \in \mathcal{W}_{\lambda}$ such that 
$$m=\widetilde{J}_{\lambda}(v_{\lambda})=\inf\limits_{v \in \mathcal{W}_{\lambda}}\widetilde{J}_{\lambda}(v) \in (-\infty, 0)\,.
$$ 
\end{lemma}
\begin{proof}
We note that $m < 0$ due to \eqref{fiber:crit:val}. Thus, we prove the remaining two steps below:\\
\noindent {\em Step~1:} Show $\inf\limits_{\mathcal{W}_{\lambda}}\widetilde{J}_{\lambda}=m > - \infty$.\\
Suppose to the contrary that  there exists  $(v_n) \subset \mathcal{W}_{\lambda}$  such that 
\begin{equation*}
\widetilde{J}_{\lambda}(v_n) = \left(\frac{1}{q}-\frac{1}{p}\right)|H_{\lambda}(v_n)|^{\frac{q}{q-p}}  \to  - \infty\,,    
\end{equation*}
that is,
\begin{equation}
\label{H:vn:-infty}
    H_{\lambda}(v_n)=\int_{\Omega^e}|\nabla v_n|^p - \lambda \int_{\Omega^e}K(x)|v_n|^p \to - \infty\,.
\end{equation}
This in turn implies that 
\begin{equation}
\label{lp:weight:infty}
\int_{\Omega^e}K(x)|v_n|^p \to \infty\,.
\end{equation}
On the other hand, \eqref{char:lambda:1:p} yields 
\begin{equation}
\label{var:char:e.v.}
   \int_{\Omega^e}|\nabla v_n|^p \geq  \lambda_1(p) \int_{\Omega^e}K(x)|v_n|^p \,.
\end{equation}
Combining \eqref{lp:weight:infty} and \eqref{var:char:e.v.}, we have
\begin{equation}
\label{grad:v_n:infty} 
   \int_{\Omega^e}|\nabla v_n|^p  \to \infty\,.
\end{equation}
Let us normalize $v_n$ in $W^{1,p}_0(\Omega^e)$, that is, define
\begin{equation*}
w_n(x):=\frac{v_n(x)}{\left(\int_{\Omega^e} |\nabla v_n|^p\right)^{\frac{1}{p}}}\,.    
\end{equation*}
Then $\|w_n\|_p=1$ and so, up to a subsequence, there exists $w_0 \in W^{1,p}_0(\Omega^e)$ such that $w_n \rightharpoonup w_0$ (weakly) in $W^{1,p}_0(\Omega^e)$. Then by \cite[Lem.~1.1]{CD2014} 
\begin{equation}
\label{compact:cd2014}
  \int_{\Omega^e}K(x)|w_n|^p \to \int_{\Omega^e}K(x)|w_0|^p\,.  
\end{equation}
Using the facts that $\|w_n\|_p=1$ and $w_n \in \mathcal{H}_{\lambda}^-$, since $v_n \in \mathcal{H}_{\lambda}^-$,  we have
\begin{equation}
\label{bound:K(x):w_n}
    \int_{\Omega^e}K(x)|w_n|^p > \frac{1}{\lambda} >0\,,
\end{equation}
for all $n \in \NN$. From \eqref{compact:cd2014} and \eqref{bound:K(x):w_n}, we deduce that 
\begin{equation}
\label{bound:K(x):w_0}
    \int_{\Omega^e}K(x)|w_0|^p \geq \frac{1}{\lambda} >0\,.
\end{equation}

On the other hand, from $v_n \in \mathcal{W}_{\lambda}$ and \eqref{grad:v_n:infty}, 
\begin{equation}
\label{norm:wn:to:0}
    \int_{\Omega^e}|\nabla w_n|^q = \frac{\int_{\Omega^e}|\nabla v_n|^q}{\left(\int_{\Omega^e}|\nabla v_n|^p\right)^{\frac{q}{p}}} = \frac{1}{\left(\int_{\Omega^e}|\nabla v_n|^p\right)^{\frac{q}{p}}} \to 0\,,
\end{equation}
that is $w_n \to 0$ a.e. in $\Omega^e$, a contradiction to \eqref{bound:K(x):w_0}. It follows that 
$(v_n)$ is bounded in  $W^{1,p}_0(\Omega^e)$. But then \eqref{var:char:e.v.} implies that $\displaystyle \int_{\Omega^e}K(x)|v_n|^p$ is bounded as well. Therefore, \eqref{H:vn:-infty} cannot hold, which proves $m=\inf\limits_{\mathcal{W}_{\lambda}}\widetilde{J}_{\lambda} > - \infty$, as desired.

\noindent {\em Step~2:} There exists $v_{\lambda} \in \mathcal{W}_{\lambda}$ such that $\widetilde{J}_{\lambda}(v_{\lambda})=m$.\\
Consider a minimizing sequence $(v_n) \subset \mathcal{W}_{\lambda}$ for $\widetilde{J}_{\lambda}$, 
\begin{equation*}
    \widetilde{J}_{\lambda}(v_n) \to  m > - \infty\,.
\end{equation*}
This is equivalent to writing that  $(v_n) \subset\mathcal{W}_{\lambda}$ is a minimizing sequence for $H_{\lambda}$,
\begin{equation*}
    \lim\limits_{n \to \infty}H_{\lambda}(v_n) = - \left(\frac{pq}{q-p}|m|\right)^{\frac{q-p}{q}} =:m'  \in (-\infty, 0)\,, 
\end{equation*}
that is, 
\begin{equation}
\label{eq:H:vn:searrow}
H_{\lambda}(v_n)= \int_{\Omega^e}|\nabla v_n|^p - \lambda  \int_{\Omega^e}K(x)|v_n|^p   \to m'\,.
\end{equation}
It follows from \eqref{eq:H:vn:searrow} that both integrals in \eqref{eq:H:vn:searrow} are either unbounded or bounded. If $\displaystyle \int_{\Omega^e}|\nabla v_n|^p \to \infty$, 
then by normalizing in $W^{1,p}_0(\Omega^e)$ and repeating the arguments of {\em Step~1}, we can reach a contradiction. Therefore, $(v_n)$ must be bounded in $W^{1,p}_0(\Omega^e)$. Passing  to a subsequence if necessary, there exists $v_{\lambda} \in W^{1,p}_0(\Omega^e)$ such that $v_n \rightharpoonup v_{\lambda}$ (weakly) in $W^{1,p}_0(\Omega^e)$.  Then by \cite[Lem.~1.1]{CD2014} 
\begin{equation}
\label{compact:cd2014:vn}
  \int_{\Omega^e}K(x)|v_n|^p \to \int_{\Omega^e}K(x)|v_{\lambda}|^p\,.  
\end{equation}
By the weak lower semicontinuity of the norm, we have 
\begin{equation}
\label{norm:wlsc}
\liminf\limits_{n \to \infty} \int_{\Omega^e}|\nabla v_n|^p \geq \int_{\Omega^e}|\nabla v_{\lambda}|^p\,.
\end{equation}
 Combining \eqref{eq:H:vn:searrow}, \eqref{compact:cd2014:vn} and \eqref{norm:wlsc}, and passing to the limit as $n \to \infty$, we get
\begin{equation}
\label{H:at:v_c}
 \int_{\Omega^e}|\nabla v_{\lambda}|^p - \lambda \int_{\Omega^e}K(x)|v_{\lambda}|^p \leq m' < 0\,.
\end{equation}
In particular, $v_{\lambda} \neq 0$ and $v_{\lambda} \in \mathcal{H}_{\lambda}^-$.

\par To complete the proof, we show that $v_{\lambda} \in \mathcal{W}_{\lambda}$. Since for any $n$, we have $v_n \in \mathcal{W}_{\lambda}$,  we may assume that $v_n \rightharpoonup v_{\lambda}$ (weakly) in $W^{1,q}_0(\Omega^e)$, passing to a subsequence if necessary, and 
\begin{equation*}
1=\liminf\limits_{n \to \infty} \int_{\Omega^e}|\nabla v_n|^q \geq \int_{\Omega^e}|\nabla v_{\lambda}|^q > 0\,.
\end{equation*}
Assume by contradiction that $\| v_{\lambda}\|_q< 1$. Take $t_{\lambda} > 1$
such that $\|t_{\lambda}v_{\lambda}\|_q=1$. Then 
\begin{equation*}
    H_{\lambda}(t_{v_\lambda}v_{\lambda})=t_{v_\lambda}^pH_{\lambda}(v_{\lambda}) \leq t_{v_\lambda}^pm' < m'
\end{equation*}
a contradiction to \eqref{H:at:v_c}. Therefore $\|v_{\lambda}\|_q=1$, i.e., $v_{\lambda} \in \mathcal{W}_{\lambda}$ and $H_{\lambda}(v_{\lambda})=m'$. Hence 
$\widetilde{J}_{\lambda}(v_{\lambda})=m$,
completing the proof of {\em Step~2}. This establishes the lemma.
\end{proof}

\section{The case: $p> q$}
\label{sec:p>q}
\begin{lemma}
    \label{lem:p>q}
    Let $p> q$ and $\lambda > \lambda_1(p)$. Then there exists $v_{\lambda} \in \mathcal{W}_{\lambda}$ such that 
    $$m=\widetilde{J}_{\lambda}(v_{\lambda})=\inf\limits_{v \in \mathcal{W}_{\lambda}}\widetilde{J}_{\lambda}(v) \in (0, \infty)\,.
    $$
\end{lemma}
\begin{proof}
We proceed again to prove this assertion in two steps. Note that, since $p>q$, it follows from \eqref{fiber:crit:val} that $0 \leq m < \infty$.\\
\noindent {\em Step~1:} Show $\inf\limits_{\mathcal{W}_{\lambda}}\widetilde{J}_{\lambda}=m > 0$.\\
Assume to the contrary that there exists a sequence $(v_n) \subset \mathcal{W}_{\lambda}$ such that 
\begin{equation}
\label{J:to:0}
    \widetilde{J}_{\lambda}(v_n) \to 0\,,
\end{equation}
that is, 
\begin{equation}
    \label{neg:H:la:infty}
    - H_{\lambda}(v_n) = \lambda \int_{\Omega^e}K(x)|v_n|^p -  \int_{\Omega^e}|\nabla v_n|^p  \to  \infty\,.
\end{equation}
This implies 
\begin{equation}
    \label{K:p:infty}
    \int_{\Omega^e}K(x)|v_n|^p \to  \infty\,.
\end{equation}
Now we normalize $(v_n)$ as follows:
\begin{equation*}
    w_n(x):= \frac{v_n(x)}{(\int_{\Omega^e}K(x)|v_n|^p)^{\frac{1}{p}}}\,.
\end{equation*}
Since $(v_n) \subset \mathcal{W}_{\lambda} \subset \mathcal{H}_{\lambda}^-$,
\begin{equation*}
    \int_{\Omega^e}|\nabla w_n|^p = \frac{\int_{\Omega^e}|\nabla v_n|^p}{\int_{\Omega^e}K(x)|v_n|^p } < \lambda 
\end{equation*}
and 
\begin{equation}
\label{q norm:w:to:0}
    \int_{\Omega^e}|\nabla w_n|^q = \frac{\int_{\Omega^e}|\nabla v_n|^q}{\left(\int_{\Omega^e}K(x)|v_n|^p\right)^{\frac{q}{p}}} = \frac{1}{\left(\int_{\Omega^e}K(x)|v_n|^p\right)^{\frac{q}{p}}} \to 0\,. 
\end{equation}
Then it follows from \eqref{est:X:norm:bound} that $(w_n)$ is bounded in $X$. Up to a subsequence, we can assume that there exists $w_0 \in X$ such that $w_n \rightharpoonup w_0$ (weakly) in $X$. By \eqref{est:Sobolev:norm:bound}, $w_n \rightharpoonup w_0$ (weakly) in $W^{1,p}_0(\Omega^e)$ and $W^{1,q}_0(\Omega^e)$ as well. The weak lower semicontinuity of the norm in $W^{1,q}_0(\Omega^e)$ combined with \eqref{q norm:w:to:0} imply that
\begin{equation*}
    \int_{\Omega^e}|\nabla w_0|^q \leq \liminf\limits_{n \to \infty}\int_{\Omega^e}|\nabla w_n|^q =0\,.
\end{equation*}
In particular, $w_0=0$ a.e. in $\Omega^e$.
\par On the other hand, since also $w_n \rightharpoonup w_0$ (weakly) in $W^{1,p}_0(\Omega^e)$ and $\displaystyle \int_{\Omega^e}K(x)|w_n|^p=1$, by \cite[Lem.~1.1]{CD2014}
\begin{equation*}
    \int_{\Omega^e}K(x)|w_n|^p \to \int_{\Omega^e}K(x)|w_0|^p\,.
\end{equation*}
Therefore $\displaystyle \int_{\Omega^e}K(x)|w_0|^p=1$, a contradiction to the fact that $w_0 =0$ a.e. in $\Omega^e$. Consequently, $\displaystyle \int_{\Omega^e}K(x)|v_n|^p$ is bounded, which in turn implies that $\displaystyle \int_{\Omega^e}|\nabla v_n|^p$ is bounded, since $v_n \in \mathcal{W}_{\lambda} \subset \mathcal{H}_{\lambda}^-$. Therefore \eqref{J:to:0} and \eqref{neg:H:la:infty} cannot occur, which proves that $m>0$.

\medskip

\noindent {\em Step~2:} There exists $v_{\lambda} \in \mathcal{W}_{\lambda}$ such that $\widetilde{J}_{\lambda}(v_{\lambda})=m$.\\
To prove this, we consider a minimizing sequence $(v_n) \subset \mathcal{W}_{\lambda}$ for $\widetilde{J}_{\lambda}$, that is, 
\begin{equation*}
   \widetilde{J}_{\lambda}(v_n) \to m > 0\,.
\end{equation*}
Since $p> q$, this is equivalent to writing that  $(v_n) \subset \mathcal{W}_{\lambda}$ is a maximizing sequence for $-H_{\lambda}$,
\begin{equation*}
    \lim\limits_{n \to \infty}-H_{\lambda}(v_n) =  \left(\frac{pq}{p-q}|m|\right)^{\frac{q-p}{q}} =:m' \in (0, \infty)\,, 
\end{equation*}
that is, 
\begin{equation}
\label{eq:H:vn:nearrow}
 \lambda  \int_{\Omega^e}K(x)|v_n|^p  - \int_{\Omega^e}|\nabla v_n|^p \to m'\,.
\end{equation}
It follows from \eqref{eq:H:vn:nearrow} that both integrals in \eqref{eq:H:vn:nearrow} are either unbounded or bounded. However, if $\displaystyle \int_{\Omega^e}|\nabla v_n|^p$ is unbounded, we can normalize as in {\em Step~1} above to get a contradiction. Therefore, $(v_n)$ is a bounded sequence in $W^{1,p}_0(\Omega^e).$ Passing to a subsequence if necessary, there exists $v_{\lambda} \in W^{1,p}_0(\Omega^e)$ such that $v_n \rightharpoonup v_{\lambda}$ (weakly) in $ W^{1,p}_0(\Omega^e)$. As before, using \cite[Lem.~1.1]{CD2014} and \eqref{norm:wlsc}, \eqref{eq:H:vn:nearrow}, we get
\begin{equation}
   \label{eq:H:vc:nearrow}
 \lambda  \int_{\Omega^e}K(x)|v_{\lambda}|^p  - \int_{\Omega^e}|\nabla v_{\lambda}|^p \geq  m' >0\,. 
\end{equation}
In particular, $v_{\lambda} \neq 0$ and $v_{\lambda} \in \mathcal{H}_{\lambda}^-$.
\par From the constraint $v_n \in \mathcal{W}_{\lambda}$, we may assume that $v_n \rightharpoonup v_{\lambda}$ (weakly) in $ W^{1,q}_0(\Omega^e)$, and
\begin{equation*}
1=\liminf\limits_{n \to \infty} \int_{\Omega^e}|\nabla v_n|^q \geq \int_{\Omega^e}|\nabla v_{\lambda}|^q > 0\,.
\end{equation*}
If $\| v_{\lambda}\|_q< 1$, we take $t_{v_\lambda} > 1$
such that $\|t_{v_\lambda}v_{\lambda}\|_q=1$. Then by \eqref{eq:H:vc:nearrow}, we get
\begin{equation*}
   - H_{\lambda}(t_{v_\lambda}v_{\lambda})=-t_{v_\lambda}^pH_{\lambda}(v_{\lambda}) = t_{v_\lambda}^pm' > m'\,,
\end{equation*}
a contradiction to \eqref{eq:H:vn:nearrow}. Therefore,  $\|v_{\lambda}\|_q=1$, i.e., $v_{\lambda} \in \mathcal{W}_{\lambda}$ and $H_{\lambda}(v_{\lambda})=m'$. Hence 
$\widetilde{J}_{\lambda}(v_{\lambda})=m$,
completing the proof of {\em Step~2}. This establishes the lemma.
\end{proof}
Note that $\mathcal{W}_{\lambda}$ is a relatively open set in $\{v \in X: G(v)=1\}$. Therefore, for each $\lambda > \lambda_1(p)$, the minimizer $v_{\lambda}$ from Lemma \ref{lem:p<q} and Lemma~\ref{lem:p>q} are, in fact, conditional critical points of $\widetilde{J}_{\lambda}$ subject to the constraint $\mathcal{W}_{\lambda}$. Then it follows from Lemma~\ref{lem:c.p:fiber} that 
\begin{equation}
\label{u:lambda:cp}
u_{\lambda}=t_{v_{\lambda}}v_{\lambda} \mbox{ is a critical point of } J_{\lambda},     
\end{equation}
where $t_{v_{\lambda}}$ is given by \eqref{def:t(v)}. Hence, $u_{\lambda}$ is a weak solution of \eqref{pde}.

\medskip

%%%%%%%%%%%%%%%%%%%%%%%%%%%%%%%%%%%%%%%%%%%%%%%%

\section{Asymptotics behavior of solutions}
\label{sec:asympt}
Here, we investigate the behavior of the weak solution $u_{\lambda}$, given by \eqref{u:lambda:cp}, as $\lambda \to \lambda_1(p)^+$ and as $\lambda \to \infty$. First, we establish some lemmas below which aid in the investigation of behaviors as $\lambda \to \lambda_1(p)^+$. 
\begin{lemma}
    \label{lem:qual:1}
    Let $\Lambda \in \RR$ and  $\Lambda > \lambda_1(p)$. Then there exists $c_{\Lambda} >0$ such that for all $\lambda \in (\lambda_1(p), \Lambda)$, 
\begin{equation}
    \label{est:weight:lp:vc:lambda}
    \int_{\Omega^e}K(x)|v_{\lambda}|^p \leq c_{\Lambda}\,.
\end{equation}    
\end{lemma}
\begin{proof}
    We proceed as in the proof of Lemma~\ref{lem:p<q}, {\em Step~1}. Assume by contradiction that there exists a sequence $(\lambda_n) \subset (\lambda_1(p), \Lambda)$ such that 
    \begin{equation}
    \label{contra:est:weight:lp:vc:lambda}
    \int_{\Omega^e}K(x)|v_{\lambda_n}|^p \to \infty\,.
\end{equation}  
Then \eqref{char:lambda:1:p} and \eqref{contra:est:weight:lp:vc:lambda} together imply 
\begin{equation*}
    \int_{\Omega^e}|\nabla v_{\lambda_n}|^p \to \infty\,.
    \end{equation*}
Now consider a normalized sequence 
\begin{equation*}
    w_n(x):=\frac{v_{\lambda_n}(x)}{\left(\int_{\Omega^e}|\nabla v_{\lambda_n}|^p\right)^{\frac{1}{p}}}\,.
\end{equation*}
Following the argument from the proof of  Lemma~\ref{lem:p<q}, {\em Step~1}, we arrive at the contradiction, proving \eqref{est:weight:lp:vc:lambda}.
\end{proof}
\begin{lemma}
    \label{lem:est:H_c:near:lambda_1}
We have
\begin{equation}
\label{est:H_c:near:lambda_1}
    H_{\lambda}(v_{\lambda}) \to 0^- \quad \mbox{ as } \lambda \to \lambda_1(p)^+\,.
\end{equation}
\end{lemma}
\begin{proof}
    Using \eqref{char:lambda:1:p}, we get
\begin{equation}
    \label{H_c:above:below}
    0 > H_{\lambda}(v_{\lambda})= \int_{\Omega^e}|\nabla v_{\lambda}|^p - \lambda \int_{\Omega^e}K(x)|v_{\lambda}|^p \geq (\lambda_1(p) - \lambda)\int_{\Omega^e}K(x)|v_{\lambda}|^p\,.
\end{equation}  
It then follows from \eqref{H_c:above:below} and Lemma~\ref{lem:qual:1} that \eqref{est:H_c:near:lambda_1} holds.
\end{proof}

\medskip

%%%%%%%%%%%%%%%%%%%%%%%%%%%%%%%%
\paragraph*{\bf Asymptotic behavior of $u_{\lambda}$:}
Now we proceed to establish the asymptotic behaviors of $u_{\lambda}$ as $\lambda \to  \infty$ and as $\lambda \to \lambda_1(p)^+$.
Clearly,
\begin{equation}
    \label{q-norm:u}    \|u_{\lambda}\|_q=\|t_{v_{\lambda}}v_{\lambda}\|_q =|t_{v_{\lambda}}| \qquad \mbox{ for}\quad v_{\lambda} \in \mathcal{W}_{\lambda}\,.
\end{equation}
On the other hand, since $v_{\lambda} \in \mathcal{W}_{\lambda}$, it follows directly from \eqref{def:t(v)}  and \eqref{fiber:on:W:lambda}  that
\begin{equation}
\label{eq:t:v_c:lambda}
   |t_{v_{\lambda}}| = \left|\frac{pq}{q-p}\right|^{\frac{1}{q}}|\widetilde{J}_{\lambda}(v_{\lambda})|^{\frac{1}{q}}=|H_{\lambda}(v_{\lambda})|^{\frac{1}{q-p}}\,,
\end{equation}
where 
\begin{equation*}
    \widetilde{J}_{\lambda}({v_{\lambda}}) = \min\limits_{v \in \mathcal{W}_{\lambda}}\widetilde{J}_{\lambda}(v)\,.
\end{equation*}
Let $\lambda > \lambda_1(p)$. It follows from \eqref{char:lambda:1:p} and Lemma~\ref{lem:Grecu:lem2} that there exists $\tilde{\varphi} \in X$ such that 
\begin{equation}
    \label{char:lambda:tilda}
 \int_{\Omega^e} |\nabla \tilde\varphi|^p=\tilde\lambda\int_{\Omega^e}K(x)|\tilde\varphi|^p   
\end{equation}
for some $\tilde\lambda \in (\lambda_1(p), \lambda)$. Without loss of generality, we may normalize $\tilde\varphi >0$ in $\Omega^e$ as $\displaystyle \int_{\Omega^e}|\nabla \tilde\varphi|^q =1$, that is, $\tilde\varphi \in \mathcal{W}_{\lambda}$. Then, we get 
\begin{align}
\label{tilde:J:v_c:lambda}
   J_{\lambda}(u_{\lambda})= \widetilde{J}_{\lambda}(v_{\lambda}) 
   &\leq \widetilde{J}_{\lambda}(\tilde\varphi) \nonumber\\&= \frac{p-q}{pq}|H_{\lambda}(\tilde\varphi)|^{\frac{q}{q-p}}
   \nonumber\\& = \frac{p-q}{pq} \left|(\tilde\lambda- \lambda) \int_{\Omega^e}K(x)|\tilde\varphi|^p\right|^{\frac{q}{q-p}}\,.
\end{align}

\par Using \eqref{tilde:J:v_c:lambda}, we deduce that for $\lambda \to \infty$, we have
\begin{enumerate}[label=(\Alph*1)]
    \item if $p < q$ then $J_{\lambda}(u_{\lambda}) \to - \infty$, and
    \item if $p > q$ then $J_{\lambda}(u_{\lambda}) \to 0$.
\end{enumerate}
It then follows from \eqref{q-norm:u}, \eqref{eq:t:v_c:lambda}, and (A1), (B1) above that as $\lambda \to \infty$, we have

\begin{enumerate}[label=(\alph*1)]
    \item if $p < q$ then $\|u_{\lambda}\|_q \to  \infty$, and
    \item if $p > q$ then  $\|u_{\lambda}\|_q  \to 0$.
\end{enumerate}

From \eqref{fiber:on:W:lambda} and Lemma~\ref{lem:est:H_c:near:lambda_1}, we deduce that we have the following asymptotics as  $\lambda  \to \lambda_1(p)^+$:  
\begin{enumerate}[label=(\Alph*2)]
    \item if $p < q$ then $J_{\lambda}(u_{\lambda}) \to 0$, and
    \item if $p > q$ then $J_{\lambda}(u_{\lambda}) \to \infty$.
\end{enumerate}

It then follows from \eqref{q-norm:u},  \eqref{eq:t:v_c:lambda}, and (A2), (B2) above that as $\lambda  \to \lambda_1(p)^+$, we have
\begin{enumerate}[label=(\alph*2)]
    \item if $p < q$ then $\|u_{\lambda}\|_{q} \to 0$, and
    \item  if $p > q$ then$\|u_{\lambda}\|_{q} \to  \infty$.
\end{enumerate}
\bigskip

%%%%%%%%%%%%%%%%%%%%%%%%%%%%%%
\section{Regularity of solutions}
\label{sec:qual:prop}
 Let us recall that proofs of \cite[Lem.~7.5 \& Lem.~7.6]{Gilbarg-Trud} are of local nature. Thus the assertions in these lemmas hold for both bounded and unbounded domains. Adopting these proofs to our situation, we get that for any $\phi \in X$, we have that $|\phi| \in X$ and  $|\nabla \phi| = |\nabla |\phi||$ a.e. in $\Omega^e$.  Moreover, for any $\phi \in X$, its truncation 
\begin{equation}
\label{eq:truncation}
\phi_M(x)= \min\{\phi(x), M\} \quad \mbox{belongs to $X$ for any} \quad M \in \RR\,.
\end{equation}

\par Thus, we can indeed assume that the minimizers found in Lemma~\ref{lem:p<q} and Lemma~\ref{lem:p>q} are nonnegative, that is, 
\begin{equation*}
    u_{\lambda}(x)=t_{v_{\lambda}}v_{\lambda}(x) \geq 0 \mbox{ a.e. in } \Omega^e\,,
\end{equation*}
and they are weak solutions of \eqref{pde}.

\par Thanks to \eqref{eq:truncation} with $M>0$, we can choose the test function in Definition~\ref{def:soln:weak} 
as in the proof of \cite[Lem.~4.5]{Dra-Kuf-Nico}. In particular, we note that the integrals involving both $p$-Laplacian and $q$-Laplacian with the test function in \cite[Lem.~4.5]{Dra-Kuf-Nico} yield nonnegtaive terms, thus in our case, we arrive at the inequality ``$\leq$" instead of the equality in \cite[Eq.~(4.47)]{Dra-Kuf-Nico}. 
 Then we can perform the Moser-type bootstrap argument exactly as in \cite[p.~176]{Dra-Kuf-Nico}, to conclude that
$u_{\lambda} \in L^{\infty}(\Omega^e)$.  

\par Consequently, we can apply the Harnack inequality \cite[Thm.~1.1]{Trud-67} (see also \cite[Thm.~1.9 \& Paragraph after Rem.~1.3]{Dra-Kuf-Nico}) to conclude that $u_{\lambda} >0$ in $\Omega^e$. Furthermore, if $\partial \Omega$ is of class $C^2$, then the regularity result up to the boundary \cite[Thm.~1]{Lieb} and \cite[p.~320]{Lieb-91} and the strong maximum principle (cf. \cite[Thm.~5.4.1]{Puc-Ser_book-07}) guarantee that $u_{\lambda} \in C_0^1(\overline{\Omega^e})$ and $\frac{\partial u_{\lambda}}{\partial \eta} <0$ on $\partial \Omega$ (see also \cite[Thm.~2.1]{Pucci-Radu-2018}).
\par Finally, due to the sign in front of the $q$-Laplacian, we can choose the test function in Definition~\ref{def:soln:weak} as in \cite[p.~257]{Serrin-64} and derive the inequality \cite[eq.~(14)]{Serrin-64}. Following Serrin's proof (\cite{Serrin-64}), we derive the inequality of \cite[p.~259]{Serrin-64}. Then the uniform decay to zero of $u_{\lambda}=u_{\lambda}(x)$ as $|x| \to \infty$ follows.
%%%%%%%%%%%%%%%%%
\section{Final Remarks}
\label{sec:final}
It is well known that the set 
\begin{equation*}
    \mathcal{N}_{\lambda}:=\left\{u \in X \setminus \{0\}: \left<J_{\lambda}'(u), u\right>=0 \right\}
\end{equation*}
is referred to as the \emph{Nehari manifold}. This manifold has been used in the study of  $(p, q)$--Laplacian eigenvalue problems in \cite{ Bob-Tana_Calc-var, Bob-Tan-2018, Bob-Tan-2022} in bounded domain setting and in \cite{Grecu2022, Mihai-Stan-2016} in unbounded domain cases. It follows from Section~\ref{sec:fiber} that if $v \in \mathcal{W}_{\lambda}$, and $t_v \neq 0$ satisfies \eqref{derivative:fiber:1}, then $u = t_v v \in \mathcal{N}_{\lambda}$. On the other hand, taking $u \in \mathcal{N}_{\lambda}$ and normalizing $v: = \frac{u}{\|u\|_q}$, we get $v \in \mathcal{W}_{\lambda}$. Therefore, there is one-to-one correspondence between $\mathcal{W}_{\lambda}$ and $\mathcal{N}_{\lambda}$. Actually, the fibered functional used by the authors in \cite{Bob-Tana_Calc-var}-\cite{Bob-Tan-2022} is similar to the one given by \eqref{fiber:crit:val}.
 Besides the functional setting
in our exterior domain case, the difference consists also of the fact
that they  minimize on the Nehari manifold $ \mathcal{N}_{\lambda}$ rather than on  $\mathcal{W}_{\lambda}$.

\par Analogous to Lemma~\ref{lem:c.p:fiber} and Lemma~\ref{lem:cond:c.p.}, one can prove that a critical point of the conditional variational problem of $J_{\lambda}$ in  $ \mathcal{N}_{\lambda}$ generates a critical point of $J_{\lambda}$ with respect to the entire space $X$ (cf. \cite[Lem.~2]{Bob-Tana_Calc-var}). Therefore, the ``Nehari manifold approach" is in fact a ``fibering method" with the special constraint given by $\mathcal{N}_{\lambda}$. 
Since the manifold  $\mathcal{N}_{\lambda}$ is defined implicitly through the energy functional, its global geometric structure is generally complicated and difficult to characterize.
\par In contrast, the advantage of introducing the manifold  $\mathcal{W}_{\lambda}$ in our work instead of  working with $\mathcal{N}_{\lambda}$ is that it allows us to  ``freeze'' the $W_0^{1,q}(\Omega^e)$ norm, when considering the constraint problem. This simplifies the energy estimates when searching for a minimizer of $\widetilde{J}_{\lambda}$ on  $\mathcal{W}_{\lambda}$ instead of  on $\mathcal{N}_{\lambda}$. Moreover, this approach yields explicit information on the dependence of both the energy and the norm of solutions in $W^{1,q}_0(\Omega^e)$ as $\lambda \to \lambda_1(p)^+$ and as $\lambda \to  \infty$.
\par More importantly, applying the fibering method offers geometric insight into the local nature of the critical points $u_{\lambda}$. Indeed, let $p < q$ and $u_{\lambda}= t_{v_{\lambda}}v_{\lambda}$ be a critical point of $J_{\lambda}$. Then $v_{\lambda} \in \mathcal{H}_{\lambda}^-$ and it follows from \eqref{fibrer} that the function 
\begin{equation}
    \label{map:t:to:J}
    t \mapsto J_{\lambda}(t\,v_{\lambda})
\end{equation}
has a local minimum at $t=t_{v_{\lambda}}$. Since $v_{\lambda}$ is also a local minimizer of the fibered functional $\widetilde{J}_{\lambda}$ with respect to $\mathcal{W}_{\lambda}$, $J_{\lambda}$ has a local minimum at $u_{\lambda}$ with respect to the  directions given by $\mathcal{N}_{\lambda}$, which are transversal to the direction given by $v_{\lambda}$. Therefore, $u_{\lambda}$ is also local minimum of $J_{\lambda}$ with respect to the entire space $X$. On the other hand, let $p>q$. Then the function \eqref{map:t:to:J} has a local maximum at $t=t_{v_{\lambda}}$. But $v_{\lambda}$ is a local minimum of $\widetilde{J}_{\lambda}$ with respect to $\mathcal{W}_{\lambda}$, and as above, $u_{\lambda}$ is a local minimum of $J_{\lambda}$ with respect to the direction given by $\mathcal{N}_{\lambda}$, which are transversal to the direction given by $v_{\lambda}$. Therefore, $J_{\lambda}$ has a saddle point character  at $u_{\lambda}$ with respect to $X$, of ``Morse index'' $1$.

\medskip
\subsection*{Acknowledgments} 
M. Chhetri was supported by a grant from the Simons Foundation 965180. 
The authors thank Prof. Vladimir Bobkov for valuable discussions on this study.

\bibliographystyle{plain}
\bibliography{references-v2}

@article {CD2014,
    AUTHOR = {Chhetri, Maya and Dr\'{a}bek, Pavel},
     TITLE = {Principal eigenvalue of {$p$}-{L}aplacian operator in exterior
              domain},
   JOURNAL = {Results Math.},
  FJOURNAL = {Results in Mathematics},
    VOLUME = {66},
      YEAR = {2014},
    NUMBER = {3-4},
     PAGES = {461--468},
      ISSN = {1422-6383},
   MRCLASS = {35J92 (35B40 35J62 35P30)},
  MRNUMBER = {3272638},
MRREVIEWER = {Leonardo Marazzi},
       DOI = {10.1007/s00025-014-0386-2},
       URL = {https://doi.org/10.1007/s00025-014-0386-2},
}

@article {Grecu2022,
    AUTHOR = {Grecu, Andrei},
     TITLE = {A perturbed eigenvalue problem in exterior domain},
   JOURNAL = {Math. Slovaca},
  FJOURNAL = {Mathematica Slovaca},
    VOLUME = {72},
      YEAR = {2022},
    NUMBER = {4},
     PAGES = {945--958},
      ISSN = {0139-9918},
   MRCLASS = {35J60 (35P99 46E35)},
  MRNUMBER = {4464447},
       DOI = {10.1515/ms-2022-0065},
       URL = {https://doi.org/10.1515/ms-2022-0065},
}

@article {KZ2022,
    AUTHOR = {Kami\'{n}ska, Anna and \.{Z}yluk, Mariusz},
     TITLE = {Uniform convexity, reflexivity, superreflexivity and {$B$}
              convexity of generalized {S}obolev spaces {$W^{1,\Phi}$}},
   JOURNAL = {J. Math. Anal. Appl.},
  FJOURNAL = {Journal of Mathematical Analysis and Applications},
    VOLUME = {509},
      YEAR = {2022},
    NUMBER = {1},
     PAGES = {Paper No. 125925, 31},
      ISSN = {0022-247X},
   MRCLASS = {46E35 (46E30)},
  MRNUMBER = {4358146},
       DOI = {10.1016/j.jmaa.2021.125925},
       URL = {https://doi.org/10.1016/j.jmaa.2021.125925},
}

@article {Mihai-Stan-2016,
    AUTHOR = {Mih\u{a}ilescu, Mihai and Stancu-Dumitru, Denisa},
     TITLE = {A perturbed eigenvalue problem on general domains},
   JOURNAL = {Ann. Funct. Anal.},
  FJOURNAL = {Annals of Functional Analysis},
    VOLUME = {7},
      YEAR = {2016},
    NUMBER = {4},
     PAGES = {529--542},
      ISSN = {2639-7390},
   MRCLASS = {35J92 (35J20 35J60 35P30 46E30 49R05)},
  MRNUMBER = {3543145},
       DOI = {10.1215/20088752-3660738},
       URL = {https://doi.org/10.1215/20088752-3660738},
}

@article {Bob-Tan-2022,
    AUTHOR = {Bobkov, Vladimir and Tanaka, Mieko},
     TITLE = {Multiplicity of positive solutions for {$(p,q)$}-{L}aplace
              equations with two parameters},
   JOURNAL = {Commun. Contemp. Math.},
  FJOURNAL = {Communications in Contemporary Mathematics},
    VOLUME = {24},
      YEAR = {2022},
    NUMBER = {3},
     PAGES = {Paper No. 2150008, 25},
      ISSN = {0219-1997},
   MRCLASS = {35P30 (35B09 35B32 35B34 35J20 35J62)},
  MRNUMBER = {4400188},
MRREVIEWER = {Daniel Maroncelli},
       DOI = {10.1142/S0219199721500085},
       URL = {https://doi.org/10.1142/S0219199721500085},
}

@article {Bob-Tan-2018,
    AUTHOR = {Bobkov, Vladimir and Tanaka, Mieko},
     TITLE = {Remarks on minimizers for {$(p,q)$}-{L}aplace equations with
              two parameters},
   JOURNAL = {Commun. Pure Appl. Anal.},
  FJOURNAL = {Communications on Pure and Applied Analysis},
    VOLUME = {17},
      YEAR = {2018},
    NUMBER = {3},
     PAGES = {1219--1253},
      ISSN = {1534-0392,1553-5258},
   MRCLASS = {35J92 (35A01 35J20)},
  MRNUMBER = {3809121},
       DOI = {10.3934/cpaa.2018059},
       URL = {https://doi.org/10.3934/cpaa.2018059},
}

@article {Bob-Tana_Calc-var,
    AUTHOR = {Bobkov, Vladimir and Tanaka, Mieko},
     TITLE = {On positive solutions for {$(p,q)$}-{L}aplace equations with
              two parameters},
   JOURNAL = {Calc. Var. Partial Differential Equations},
  FJOURNAL = {Calculus of Variations and Partial Differential Equations},
    VOLUME = {54},
      YEAR = {2015},
    NUMBER = {3},
     PAGES = {3277--3301},
      ISSN = {0944-2669,1432-0835},
   MRCLASS = {35J92 (35A01 35B09 35J20 35J62 35P30)},
  MRNUMBER = {3412411},
MRREVIEWER = {Vicen\c{t}iu\ D.\ R\u{a}dulescu},
       DOI = {10.1007/s00526-015-0903-5},
       URL = {https://doi.org/10.1007/s00526-015-0903-5},
}

@article {Serrin-64,
    AUTHOR = {Serrin, James},
     TITLE = {Local behavior of solutions of quasi-linear equations},
   JOURNAL = {Acta Math.},
  FJOURNAL = {Acta Mathematica},
    VOLUME = {111},
      YEAR = {1964},
     PAGES = {247--302},
      ISSN = {0001-5962,1871-2509},
   MRCLASS = {35.47},
  MRNUMBER = {170096},
MRREVIEWER = {C.\ B.\ Morrey, Jr.},
       DOI = {10.1007/BF02391014},
       URL = {https://doi.org/10.1007/BF02391014},
}

@article {pohozaev-99,
    AUTHOR = {Poho\v{z}aev, S. I.},
     TITLE = {The fibering method and its applications to nonlinear boundary
              value problem},
   JOURNAL = {Rend. Istit. Mat. Univ. Trieste},
  FJOURNAL = {Rendiconti dell'Istituto di Matematica dell'Universit\`a di
              Trieste. An International Journal of Mathematics},
    VOLUME = {31},
      YEAR = {1999},
    NUMBER = {1-2},
     PAGES = {235--305},
      ISSN = {0049-4704,2464-8728},
   MRCLASS = {58E05 (35J20 35J65 35J70 47J20)},
  MRNUMBER = {1763253},
MRREVIEWER = {Lorenzo\ Pisani},
}

@incollection {pohozaev-97,
    AUTHOR = {Pohozaev, S. I.},
     TITLE = {The fibering method in nonlinear variational problems},
 BOOKTITLE = {Topological and variational methods for nonlinear boundary
              value problems ({C}hol\'{\i}n, 1995)},
    SERIES = {Pitman Res. Notes Math. Ser.},
    VOLUME = {365},
     PAGES = {35--88},
 PUBLISHER = {Longman, Harlow},
      YEAR = {1997},
      ISBN = {0-582-30921-2},
   MRCLASS = {58E05 (35J65 47H15)},
  MRNUMBER = {1478747},
MRREVIEWER = {Jes\'{u}s\ Hern\'{a}ndez},
}

@book {Puc-Ser_book-07,
    AUTHOR = {Pucci, Patrizia and Serrin, James},
     TITLE = {The maximum principle},
    SERIES = {Progress in Nonlinear Differential Equations and their
              Applications},
    VOLUME = {73},
 PUBLISHER = {Birkh\"{a}user Verlag, Basel},
      YEAR = {2007},
     PAGES = {x+235},
      ISBN = {978-3-7643-8144-8},
   MRCLASS = {35-02 (34B15 35B50)},
  MRNUMBER = {2356201},
MRREVIEWER = {Rodney\ Josu\'{e}\ Biezuner},
}

@book {Gilbarg-Trud,
    AUTHOR = {Gilbarg, David and Trudinger, Neil S.},
     TITLE = {Elliptic partial differential equations of second order},
    SERIES = {Classics in Mathematics},
   EDITION = {1998},
 PUBLISHER = {Springer-Verlag, Berlin},
      YEAR = {2001},
     PAGES = {xiv+517},
      ISBN = {3-540-41160-7},
   MRCLASS = {35-02 (35Jxx)},
  MRNUMBER = {1814364},
}

@article {Trud-67,
    AUTHOR = {Trudinger, Neil S.},
     TITLE = {On {H}arnack type inequalities and their application to
              quasilinear elliptic equations},
   JOURNAL = {Comm. Pure Appl. Math.},
  FJOURNAL = {Communications on Pure and Applied Mathematics},
    VOLUME = {20},
      YEAR = {1967},
     PAGES = {721--747},
      ISSN = {0010-3640,1097-0312},
   MRCLASS = {35.47},
  MRNUMBER = {226198},
MRREVIEWER = {A.\ Kufner},
       DOI = {10.1002/cpa.3160200406},
       URL = {https://doi.org/10.1002/cpa.3160200406},
}

@book {Dra-Kuf-Nico,
    AUTHOR = {Dr\'{a}bek, Pavel and Kufner, Alois and Nicolosi, Francesco},
     TITLE = {Quasilinear elliptic equations with degenerations and
              singularities},
    SERIES = {De Gruyter Series in Nonlinear Analysis and Applications},
    VOLUME = {5},
 PUBLISHER = {Walter de Gruyter \& Co., Berlin},
      YEAR = {1997},
     PAGES = {xii+219},
      ISBN = {3-11-015490-0},
   MRCLASS = {35J65 (35J70 47H15 47N20)},
  MRNUMBER = {1460729},
MRREVIEWER = {Isabella\ Birindelli},
       DOI = {10.1515/9783110804775},
       URL = {https://doi.org/10.1515/9783110804775},
}

@article {Lieb,
    AUTHOR = {Lieberman, Gary M.},
     TITLE = {Boundary regularity for solutions of degenerate elliptic
              equations},
   JOURNAL = {Nonlinear Anal.},
  FJOURNAL = {Nonlinear Analysis. Theory, Methods \& Applications. An
              International Multidisciplinary Journal},
    VOLUME = {12},
      YEAR = {1988},
    NUMBER = {11},
     PAGES = {1203--1219},
      ISSN = {0362-546X,1873-5215},
   MRCLASS = {35J70 (35B65)},
  MRNUMBER = {969499},
MRREVIEWER = {Zuchi\ Chen},
       DOI = {10.1016/0362-546X(88)90053-3},
       URL = {https://doi.org/10.1016/0362-546X(88)90053-3},
}

@article {Lieb-91,
    AUTHOR = {Lieberman, Gary M.},
     TITLE = {The natural generalization of the natural conditions of Ladyzhenskaya and Uraletseva for elliptic
              equations},
   JOURNAL = {Comm. Partial Differential Equations},
  FJOURNAL = {Communications in Partial Differential Equations},
    VOLUME = {16},
      YEAR = {1991},
    NUMBER = {2-3},
     PAGES = {311--361},
      ISSN = {0360-5302,1532-4133},
   MRCLASS = {35J60 (35B65)},
  MRNUMBER = {1104103},
MRREVIEWER = {M.\ Biroli},
       DOI = {10.1080/03605309108820761},
       URL = {https://doi.org/10.1080/03605309108820761},
}

@article {Pohojaev-1979,
    AUTHOR = {Poho\v{z}aev, S. I.},
     TITLE = {An approach to nonlinear equations},
   JOURNAL = {Dokl. Akad. Nauk SSSR},
  FJOURNAL = {Doklady Akademii Nauk SSSR},
    VOLUME = {247},
      YEAR = {1979},
    NUMBER = {6},
     PAGES = {1327--1331},
      ISSN = {0002-3264},
   MRCLASS = {47H15 (35G20 58E30)},
  MRNUMBER = {550349},
MRREVIEWER = {Valeri\ Obukhovski\u{\i}},
}

@article {Pohojaev-1988,
    AUTHOR = {Pokhozhaev, S. I.},
     TITLE = {On a constructive method of the calculus of variations},
   JOURNAL = {Dokl. Akad. Nauk SSSR},
  FJOURNAL = {Doklady Akademii Nauk SSSR},
    VOLUME = {298},
      YEAR = {1988},
    NUMBER = {6},
     PAGES = {1330--1333},
      ISSN = {0002-3264},
   MRCLASS = {49B21 (58E05 58E99)},
  MRNUMBER = {947797},
MRREVIEWER = {Pavel\ Dr\'{a}bek},
}

@incollection {Pohozaev-1990,
    AUTHOR = {Pokhozhaev, S. I.},
     TITLE = {The fibration method for solving nonlinear boundary value
              problems},
      NOTE = {Translated in Proc. Steklov Inst. Math. {\bf 1992}, no. 3,
              157--173,
              Differential equations and function spaces (Russian)},
   JOURNAL = {Trudy Mat. Inst. Steklov.},
  FJOURNAL = {Akademiya Nauk SSSR. Trudy Matematicheskogo Instituta imeni V.
              A. Steklova},
    VOLUME = {192},
      YEAR = {1990},
     PAGES = {146--163},
      ISSN = {0371-9685},
   MRCLASS = {47H15 (35J65 47N20 58E05)},
  MRNUMBER = {1097896},
MRREVIEWER = {D.\ Pascali},
}

@article {Drabek-Pohozaev,
    AUTHOR = {Dr\'{a}bek, Pavel and Pohozaev, Stanislav I.},
     TITLE = {Positive solutions for the {$p$}-{L}aplacian: application of
              the fibering method},
   JOURNAL = {Proc. Roy. Soc. Edinburgh Sect. A},
  FJOURNAL = {Proceedings of the Royal Society of Edinburgh. Section A.
              Mathematics},
    VOLUME = {127},
      YEAR = {1997},
    NUMBER = {4},
     PAGES = {703--726},
      ISSN = {0308-2105,1473-7124},
   MRCLASS = {35J60 (35B05 35B32 35J70)},
  MRNUMBER = {1465416},
MRREVIEWER = {P.\ Lindqvist},
       DOI = {10.1017/S0308210500023787},
       URL = {https://doi.org/10.1017/S0308210500023787},
}

@article {Chen-Hu-Zhao-2003,
    AUTHOR = {Chen, Shutao and Hu, Changying and Zhao, Charles Xuejin},
     TITLE = {Uniform rotundity of {O}rlicz-{S}obolev spaces},
   JOURNAL = {Soochow J. Math.},
  FJOURNAL = {Soochow Journal of Mathematics},
    VOLUME = {29},
      YEAR = {2003},
    NUMBER = {3},
     PAGES = {299--312},
      ISSN = {0250-3255},
   MRCLASS = {46E35 (46B03 46B20 47L10)},
  MRNUMBER = {2003473},
MRREVIEWER = {Wolfgang Lusky},
}

@article {Pucci-Radu-2018,
    AUTHOR = {Pucci, Patrizia and R\u{a}dulescu, Vicen\c{t}iu D.},
     TITLE = {The maximum principle with lack of monotonicity},
   JOURNAL = {Electron. J. Qual. Theory Differ. Equ.},
  FJOURNAL = {Electronic Journal of Qualitative Theory of Differential
              Equations},
      YEAR = {2018},
     PAGES = {Paper No. 58, 11},
      ISSN = {1417-3875},
   MRCLASS = {35J92 (35B50 35B51 35R45)},
  MRNUMBER = {3827996},
       DOI = {10.14232/ejqtde.2018.1.58},
       URL = {https://doi.org/10.14232/ejqtde.2018.1.58},
}
\end{document}